\documentclass[12pt]{amsart}
\usepackage{amsmath}
\usepackage{amsfonts}
\begin{document}
\numberwithin{equation}{section}

\def\1#1{\overline{#1}}
\def\2#1{\widetilde{#1}}
\def\3#1{\widehat{#1}}
\def\4#1{\mathbb{#1}}
\def\5#1{\frak{#1}}
\def\6#1{{\mathcal{#1}}}

\newcommand{\de}{\partial}
\newcommand{\R}{\mathbb R}
\newcommand{\Ha}{\mathbb H}
\newcommand{\al}{\alpha}
\newcommand{\tr}{\widetilde{\rho}}
\newcommand{\tz}{\widetilde{\zeta}}
\newcommand{\tk}{\widetilde{C}}
\newcommand{\tv}{\widetilde{\varphi}}
\newcommand{\hv}{\hat{\varphi}}
\newcommand{\tu}{\tilde{u}}
\newcommand{\tF}{\tilde{F}}
\newcommand{\tG}{\tilde{G}}
\newcommand{\debar}{\overline{\de}}
\newcommand{\Z}{\mathbb Z}
\newcommand{\C}{\mathbb C}
\newcommand{\Po}{\mathbb P}
\newcommand{\zbar}{\overline{z}}
\newcommand{\G}{\mathcal{G}}
\newcommand{\So}{\mathcal{S}}
\newcommand{\Ko}{\mathcal{K}}
\newcommand{\U}{\mathcal{U}}
\newcommand{\B}{\mathbb B}
\newcommand{\oB}{\overline{\mathbb B}}
\newcommand{\Cur}{\mathcal D}
\newcommand{\Dis}{\mathcal Dis}
\newcommand{\Levi}{\mathcal L}
\newcommand{\SP}{\mathcal SP}
\newcommand{\Sp}{\mathcal Q}
\newcommand{\A}{\mathcal O^{k+\alpha}(\overline{\mathbb D},\C^n)}
\newcommand{\CA}{\mathcal C^{k+\alpha}(\de{\mathbb D},\C^n)}
\newcommand{\Ma}{\mathcal M}
\newcommand{\Ac}{\mathcal O^{k+\alpha}(\overline{\mathbb D},\C^{n}\times\C^{n-1})}
\newcommand{\Acc}{\mathcal O^{k-1+\alpha}(\overline{\mathbb D},\C)}
\newcommand{\Acr}{\mathcal O^{k+\alpha}(\overline{\mathbb D},\R^{n})}
\newcommand{\Co}{\mathcal C}
\newcommand{\Hol}{{\sf Hol}(\mathbb H, \mathbb C)}
\newcommand{\Aut}{{\sf Aut}(\mathbb D)}
\newcommand{\D}{\mathbb D}
\newcommand{\oD}{\overline{\mathbb D}}
\newcommand{\oX}{\overline{X}}
\newcommand{\loc}{L^1_{\rm{loc}}}
\newcommand{\la}{\langle}
\newcommand{\ra}{\rangle}
\newcommand{\thh}{\tilde{h}}
\newcommand{\N}{\mathbb N}
\newcommand{\kd}{\kappa_D}
\newcommand{\Hr}{\mathbb H}
\newcommand{\ps}{{\sf Psh}}
\newcommand{\Hess}{{\sf Hess}}
\newcommand{\subh}{{\sf subh}}
\newcommand{\harm}{{\sf harm}}
\newcommand{\ph}{{\sf Ph}}
\newcommand{\tl}{\tilde{\lambda}}
\newcommand{\gdot}{\stackrel{\cdot}{g}}
\newcommand{\gddot}{\stackrel{\cdot\cdot}{g}}
\newcommand{\fdot}{\stackrel{\cdot}{f}}
\newcommand{\fddot}{\stackrel{\cdot\cdot}{f}}
\def\v{\varphi}
\def\Re{{\sf Re}\,}
\def\Im{{\sf Im}\,}


\def\Label#1{\label{#1}}


\def\cn{{\C^n}}
\def\cnn{{\C^{n'}}}
\def\ocn{\2{\C^n}}
\def\ocnn{\2{\C^{n'}}}
\def\je{{\6J}}
\def\jep{{\6J}_{p,p'}}
\def\th{\tilde{h}}


\def\dist{{\rm dist}}
\def\const{{\rm const}}
\def\rk{{\rm rank\,}}
\def\id{{\sf id}}
\def\aut{{\sf aut}}
\def\Aut{{\sf Aut}}
\def\CR{{\rm CR}}
\def\GL{{\sf GL}}
\def\Re{{\sf Re}\,}
\def\Im{{\sf Im}\,}
\def\U{{\sf U}}

\def\la{\langle}
\def\ra{\rangle}

\emergencystretch15pt \frenchspacing

\newtheorem{theorem}{Theorem}[section]
\newtheorem{lemma}[theorem]{Lemma}
\newtheorem{proposition}[theorem]{Proposition}
\newtheorem{corollary}[theorem]{Corollary}

\theoremstyle{definition}
\newtheorem{definition}[theorem]{Definition}
\newtheorem{example}[theorem]{Example}

\theoremstyle{remark}
\newtheorem{remark}[theorem]{Remark}
\numberwithin{equation}{section}

\title[Semigroups versus evolution families]{Semigroups versus evolution families in the Loewner theory}
\author[F. Bracci]{Filippo Bracci}
\address{Dipartimento Di Matematica, Universit\`{a} Di Roma \textquotedblleft Tor
Vergata\textquotedblright, Via Della Ricerca Scientifica 1,
00133, Roma, Italy. } \email{fbracci@mat.uniroma2.it}
\author[M. D. Contreras]{Manuel D. Contreras}

\author[S. D\'{\i}az-Madrigal]{Santiago D\'{\i}az-Madrigal}
\address{Camino de los Descubrimientos, s/n\\
Departamento de Matem\'{a}tica Aplicada II\\
Escuela T\'{e}cnica Superior de Ingenieros\\
Universidad de Sevilla\\
Sevilla, 41092\\
Spain.}\email{contreras@us.es} \email{madrigal@us.es}
\date{\today }


\thanks{Partially supported by the \textit{Ministerio
de Ciencia e Innovaci\'on} and the European Union (FEDER),
project MTM2006-14449-C02-01, by \textit{La Consejer\'{\i}a de
Educaci\'{o}n y Ciencia de la Junta de Andaluc\'{\i}a} and by
the European Science Foundation Research Networking Programme
HCAA}

\begin{abstract}
We show that an evolution family of the unit disc is commuting
if and only if the associated Herglotz vector field has
separated variables. This is the case  if and only if the
evolution family   comes from a semigroup of holomorphic
self-maps of the disc.
\end{abstract}

\maketitle


\section{Introduction}

In 1923,  Loewner \cite{Loewner} introduced a differential
equation to study some extremal problems in the theory of
univalent functions, later developed mainly by Kufarev and
Pommerenke. Such equation is nowadays known as the radial
Loewner equation and it has been used to obtain many
fundamental results such as distortion theorems, growth
theorems, rotation theorems (see, {\sl e.g.}
\cite{Pommerenke}). In particular, Loewner's radial equation
was a key ingredient in the proof of Bieberbach's conjecture by
de Branges in 1985. In the last two decades, many
mathematicians  have considered and studied a variant of that
equation which is called the chordal Loewner differential
equation. Such a theory, especially the stochastic version of
it, turned out to be useful for solving famous open
conjectures. For instance, Lawler, Schramm and Werner solved
the Mandelbrot's conjecture about the Haussdorf dimension of
the Brownian frontier. For further details, we refer the reader
to \cite{Ma-Ro} and references therein.

Recently, the  authors  and  Gumenyuk  developed a theory which
unifies and extends both the radial and the chordal Loewner
equations \cite{BCM1}, \cite{Contreras-Diaz-Pavel}. Indeed,
this theory carries out to complex hyperbolic manifolds
\cite{BCM2}.

Loewner's theory studies the relationships among three notions:
Herglotz vector fields, evolution families and Loewner chains.
Roughly speaking, a Herglotz vector field $G(z,t)$ is a
Carath\`eodory vector field  such that $G(\cdot, t)$ is
semicomplete for almost every $t\geq 0$ (see Definition
\ref{DH}). An evolution family $(\v_{s,t})$ is a family of
holomorphic self-maps of the unit disc $\D$ satisfying some
algebraic relations in $s,t$ and some regularity hypotheses
(see Definition \ref{DE}). Finally, a Loewner chain $(f_t)$ is
a family of univalent mappings on the unit disc with increasing
ranges satisfying some regularity assumptions (see Definition
\ref{DL}).

The three objects are related by the following Loewner
differential equations:
\[
\frac{\de \v_{s,t}(z)}{\de t}=G(\v_{s,t}(z),t), \quad \frac{\de f_t(z)}{\de t}=-f_t'(z)G(z,t), \quad f_s(z)=f_t(\v_{s,t}(z)).
\]
In \cite{BCM1} it is proved that there is a one-to-one
correspondence between evolution families and Herglotz vector
fields, while in general Loewner chains are not uniquely
associated with Herglotz vector fields
\cite{Contreras-Diaz-Pavel}.

Examples of evolution families are given by semigroups of
holomorphic self-maps of the unit disc. Namely, if $(\Phi_t)$
is a semigroup (see Subsection \ref{semig}) then setting
$\v_{s,t}:=\Phi_{t-s}$ for $0\leq s\leq t<+\infty$ we obtain an
evolution family \cite[Example 3.4]{BCM1}. The associated
Herglotz vector field $G(z,t)$ does not depend on $t$ and it is
actually the infinitesimal generator of the semigroup. More
generally, if $\lambda:[0,+\infty)\to [0,+\infty)$ is an
increasing absolutely continuous function  then
$(\Phi_{\lambda(t)-\lambda(s)})$ is an evolution family whose
Herglotz vector field is {\sl splitting}, in the sense that
$G(z,t)=\stackrel{\bullet}{\lambda}(t) \tG(z)$ with $\tG$ being
the infinitesimal generator of the semigroup. Note that in such
cases the evolution family is commuting, namely every element
of the family commutes with each other.

The aim of the present paper is to characterize Herglotz vector
fields which are splitting  (see Definition \ref{splitting}).
The main result of this paper is the following

\begin{theorem}\label{main}
Let $G(z,t)$ be a Herglotz vector field and let $(\v_{s,t})$ be
its associated evolution family. Then $G(z,t)$ is splitting if
and only if $(\v_{s,t})$ is commuting.
\end{theorem}

Such a result is proved in Theorems \ref{m1} and \ref{maindue}.
Moreover, we show in Proposition \ref{uno} that a Herglotz
vector field has an associated Loewner chain of a particular
affine form if and only if it is splitting. Also, in Section
\ref{split} we describe splitting Herglotz vector fields
according to the dynamical properties of related semigroups and
we provide their Berkson-Porta like decomposition.

Finally, in Section \ref{rever} we introduce the notion of {\sl
reversing} evolution family, a natural and weaker  notion of
commuting, and we show that reversing evolution family are
commuting in many cases (see Theorems \ref{revcom} and
\ref{revcom2}).

\section{Preliminaries}\label{prel}

\subsection{Iteration theory} Let $\D:=\{\zeta\in \C: |\zeta|<1\}$ be the unit disc
of $\C$. A holomorphic function $f:\D\to \D$ such that $f\neq
\id$ has at most one fixed point in $\D$. If $f$ has a fixed
point $\tau\in\D$, then $f$ is called {\sl elliptic} and such a
point is called the {\sl Denjoy-Wolff point} of $f$. In case
$f$ is not an elliptic automorphism the sequence of iterates
$\{f^{\circ n}\}$ converges uniformly on compacta to the
constant function $z\mapsto \tau$.

In case $f$ has no fixed points in $\D$ then there exists a
unique point $\tau\in \de \D$, called again the {\sl
Denjoy-Wolff point} of $f$, such that $\{f^{\circ n}\}$
converges uniformly on compacta to the constant function
$z\mapsto \tau$. Moreover, $\angle \lim_{z\to \tau}f(z)=\tau$
and $\angle \lim_{z\to \tau}f'(z)=\alpha_f$, with $\alpha_f\in
(0,1]$ (here, as customary, $\angle\lim_{z\to \tau}$ means
angular limit). The function $f$ is said {\sl hyperbolic} if
$\alpha_f<1$ and {\sl parabolic} if $\alpha_f=1$ (see, {\sl
e.g.} \cite{Abate}).

If $f$ is parabolic, it is said {\sl of zero parabolic step} if
for some---and hence any---$z\in \D$ it follows
\[
\lim_{n\to \infty}\omega (f^{\circ n}(z), f^{(\circ n+1)}(z))=0,
\]
where $\omega$ is the Poincar\'e distance of $\D$.

The following result about centralizers of holomorphic
self-maps of the disc is true in a more general context without
assuming injectivity, but here we only need in the following
simple form.

\begin{lemma}\label{centro}
Let $f:\D \to \D$ be a univalent function, $f\neq \id$. Let
$C(f):=\{g:\D \to \D: f\circ g=g\circ f\}$ be the centralizer.
Then
\begin{enumerate}
  \item If $f$ is a hyperbolic automorphism with distinct fixed points $\tau,\tau'\in \de \D$ then $C(f)$ is
  abelian and for all $g\in C(f)$ it follows that $g$ is a
  hyperbolic automorphism with fixed points $\tau, \tau'$.
  \item If $f$ is not an automorphism and it is elliptic or
  hyperbolic then $C(f)$ is abelian.
  \item If $f$ is parabolic of zero hyperbolic step then $C(f)$
  is abelian.
\end{enumerate}
\end{lemma}

\begin{proof}
(1) It is Heins' theorem \cite{He2}.

(2) It is due to Cowen \cite[Corollary 4.2]{Cow}. However, in
the hyperbolic case, one can get a simpler proof arguing along
the lines  of (3) below using the uniqueness of the
intertwining function for hyperbolic mappings proved in
\cite{Br1}.

(3) Let $\sigma:\D\to \C$ be univalent and such that $\sigma
\circ f= \sigma+1$. Such a map $\sigma$ exists and it is unique
in the sense that if $\tilde{\sigma}:\D\to \C$ is another
univalent map such that $\tilde{\sigma} \circ f=
\tilde{\sigma}+1$ then there exists $\lambda$ such that
$\tilde{\sigma}=\sigma+\lambda$ (\cite[Theorem 3.1]{CoDiPo}).
Let $g\in C(f)$ and write $\tilde{\sigma}:=\sigma \circ g$. It
follows that
\[
\tilde{\sigma}\circ f= \sigma \circ f \circ g= \sigma\circ g +1 =\tilde{\sigma}+1,
\]
hence $\sigma \circ g=\sigma+\lambda_g$ for some $\lambda_g \in
\C$. Now, if $g,g'\in C(f)$ then
\[
\sigma \circ g \circ g'=\sigma +\lambda_g +\lambda_{g'}=\sigma +\lambda_{g'}+\lambda_g=\sigma\circ g'\circ g.
\]
Being $\sigma$ univalent, it follows that $g\circ g'=g'\circ
g$, as wanted.
\end{proof}

Finally, we recall that a point $p\in \de \D$ is said to be a
{\sl boundary repelling fixed point} for a holomorphic map
$f:\D\to \D$ if $\lim_{r\to 1^+}f(rp)=p$ and $\lim_{r\to
1^+}f'(rp)=C$ with $C\in (1,+\infty)$.

\subsection{Semigroups}\label{semig} A semigroup $(\Phi_t)$ of holomorphic
self-maps of $\D$ is a continuous homomorphism between the
additive semigroup $(\R^+, +)$ of positive real numbers and the
semigroup $({\sf Hol}(\D,\D),\circ)$ of holomorphic self-maps
of $\D$ with respect to the composition, endowed with the
topology of uniform convergence on compacta.

By Berkson-Porta's theorem \cite{Berkson-Porta}, if $(\Phi_t)$
is a semigroup in ${\sf Hol}(\D,\D)$ then $t\mapsto \Phi_t(z)$
is analytic and there exists a unique holomorphic vector field
$F:\D\to \C$ such that $\frac{\de \Phi_t(z)}{\de
t}=F(\Phi_t(z))$. Such a vector field $F$ is semicomplete and
it is called the {\sl infinitesimal generator} of $(\Phi_t)$.
Conversely, any semicomplete holomorphic vector field in $\D$
generates a semigroup in ${\sf Hol}(\D,\D)$.

Let $F\not\equiv 0$ be an infinitesimal generator with
associated semigroup $(\Phi_t)$. Then there exists a unique
$\tau\in\oD$ and a unique $p:\D\to \C$ holomorphic with $\Re
p(z)\geq 0$ such that $F(z)=(z-\tau)(\overline{\tau}z-1)p(z)$.
Such a formula is the well renowned {\sl Berkson-Porta
formula}.

The point $\tau$ in the Berkson-Porta formula turns out to be
the common Denjoy-Wolff point of $\Phi_t$ for all $t\geq 0$.
Moreover, if $\tau\in \de\D$ it follows $\angle\lim_{z\to
\tau}\Phi_t'(z)=e^{\beta t}$ for some $\beta\leq 0$, where
$\beta=0$ if and only if $\Phi_t$ is parabolic for some--hence
any--$t>0$.

A {\sl boundary repelling fixed point} for a semigroup
$(\Phi_t)$ is a point $p\in \de \D$ which is a boundary
repelling fixed point for one---and hence any---$\Phi_t$, $t>0$
\cite{CoDiPo2}. Moreover, if $p\in \de\D$ is a boundary
repelling fixed point for $(\Phi_t)$, then there exists
$\beta>0$ such that $\lim_{r\to 1^+}\Phi_t(rp)=e^{t\beta}$ (see
\cite{Contreras-Diaz-Pommerenke:Scand}).

The proof of following proposition is in \cite{Berkson-Porta}
and \cite{Siskakis-tesis}.

\begin{proposition}
\label{Unival-VectorField} Let $(\Phi_{t})$ be a non-trivial
semigroup in $\mathbb{D}$ with infinitesimal generator $G$.
Then there exists a unique univalent function $h:\D\to\C$,
called the {\sl K\"onigs function} of $(\Phi_t)$ such that
\begin{enumerate}
\item If $(\Phi_{t})$ has Denjoy-Wolff point $\tau \in \mathbb{D}$ then
$h(\tau)=0$, $h'(\tau)=1$ and
$h(\Phi_t(z))=e^{G(\tau)t}h(z)$ for all $t\geq 0$.
Moreover, $h$ is the unique holomorphic function from
$\mathbb{D}$ into $\mathbb{C}$ such that
\begin{enumerate}
\item[(i)] $h^{\prime }(z)\neq 0,$ for every $z\in \mathbb{D},$
\item[(ii)] $h(\tau )=0$ and $h^{\prime }(\tau )=1,$
\item[(iii)] $h^{\prime }(z)G(z)=G^{\prime }(\tau )h(z),$ for
every $z\in \mathbb{D}.$
\end{enumerate}
\item If $(\Phi_{t})$ has Denjoy-Wolff point $\tau
\in \partial \mathbb{D}$ then $h(0)=0$ and
$h(\Phi_t(z))=h(z)+t$ for all $t\geq 0$. Moreover, $h$ is
the unique holomorphic function from $\mathbb{D}$ into
$\mathbb{C}$ such  that:
\begin{enumerate}
\item[(i)] $h(0)=0,$
\item[(ii)] $h^{\prime }(z)G(z)=1,$ for every $z\in \mathbb{D}.$
\end{enumerate}
\end{enumerate}
\end{proposition}

\subsection{Loewner theory} The three main objects of the theory are Herglotz vector
fields, evolution families and Loewner chains. We give here the
actual general definitions from \cite{BCM1} and
\cite{Contreras-Diaz-Pavel} which include the classical  radial
and chordal cases.

\begin{definition}
\label{DH} Let $d\in [1,+\infty]$. A {\sl Herglotz
 vector field of order $d$} on the unit disc
$\mathbb{D}$ is a function
$G:\mathbb{D}\times\lbrack0,+\infty)\rightarrow \mathbb{C}$
with the following properties:

\begin{enumerate}
\item[H1.] For all $z\in\mathbb{D},$ the function $\lbrack
0,+\infty)\ni t\mapsto G(z,t)$ is measurable;

\item[H2.] For all $t\in\lbrack0,+\infty),$ the function $
\mathbb{D}\ni z\mapsto G(z,t)$ is holomorphic;

\item[H3.] For any compact set $K\subset\mathbb{D}$ and for all $T>0$ there
exists a non-negative function $k_{K,T}\in
L^{d}([0,T],\mathbb{R})$ such that
\[
|G(z,t)|\leq k_{K,T}(t)
\]
for all $z\in K$ and for almost every $t\in\lbrack0,T].$
\item[H4.] For almost every $t\in [0,+\infty)$ it follows
$G(\cdot, t)$ is an infinitesimal generator.
\end{enumerate}
\end{definition}

In \cite[Theorem 4.8]{BCM1} it is proved that any Herglotz
vector field $G(z,t)$ has a decomposition by means of a
Berkson-Porta like formula, namely,
$G(z,t)=(z-\tau(t))(\overline{\tau(t)}z-1)p(z,t)$, where
$\tau:[0,+\infty)\to \oD$ is a measurable function and
$p:\D\times [0,+\infty)\to \C$ has the property that for all
$z\in\mathbb{D},$ the function $\lbrack0,+\infty )\ni t \mapsto
p(z,t)\in\mathbb{C}$ belongs to
$L_{loc}^{d}([0,+\infty),\mathbb{C})$; for all
$t\in\lbrack0,+\infty),$ the function $\mathbb{D}\ni z \mapsto
p(z,t)\in\mathbb{C}$ is holomorphic; for all $z\in\mathbb{D}$
and for all $t\in\lbrack0,+\infty),$ we have $\Re p(z,t)\geq0.$
The data $(\tau(t), p(z,t))$ are called the {\sl Berkson-Porta
data} of $G(z,t)$ and they are essentially unique, in the sense
that $p(z,t)$ is unique up to a zero measure set in $t$ and
$\tau(t)$ is unique if $p(z,t)\not\equiv 0$.

\begin{definition}\label{splitting}
A Herglotz vector field $G(z,t)$ of order $d\in [1,+\infty]$ is
said to be  {\sl splitting} if there exists an infinitesimal
generator $\tG$ and a function $g\in L^d_{\sf
loc}([0,+\infty),\C)$ such that $G(z,t)=g(t)\tG(z)$ for all
$z\in\D$ and almost every $t\in [0,+\infty)$.
\end{definition}

Now we recall the definition of evolution families.

\begin{definition}\label{DE}
A family $(\varphi_{s,t})_{0\leq s\leq t<+\infty}$ of
holomorphic self-maps of the unit disc  is an {\sl evolution
family of order $d$} with $d\in [1,+\infty]$ if

\begin{enumerate}
\item[EF1.] $\varphi_{s,s}=id_{\mathbb{D}},$

\item[EF2.] $\varphi_{s,t}=\varphi_{u,t}\circ\varphi_{s,u}$ for all $0\leq
s\leq u\leq t<+\infty,$

\item[EF3.] for all $z\in\mathbb{D}$ and for all $T>0$ there exists a
non-negative function $k_{z,T}\in L^{d}([0,T],\mathbb{R})$
such that
\[
|\varphi_{s,u}(z)-\varphi_{s,t}(z)|\leq\int_{u}^{t}k_{z,T}(\xi)d\xi
\]
for all $0\leq s\leq u\leq t\leq T.$
\end{enumerate}
\end{definition}

The elements of evolution families are univalent
\cite[Corollary 6.3]{BCM1}.

In \cite[Theorem 1.1, Theorem 6.6]{BCM1} it is proved that
there is a one-to-one correspondence between evolution families
and Herglotz vector fields:

\begin{theorem}\label{LODE}
For any evolution family $(\v_{s,t})$ of order $d\geq 1$ in the
unit disc there exists a unique (up to changing on zero measure
set in $t$) Herglotz vector field $G(z,t)$ of order $d$ such
that for all $z\in \D$
\begin{equation}\label{main-eq}
\frac{\de \v_{s,t}(z)}{\de t}=G(\v_{s,t}(z),t) \quad \hbox{a.e.
$t\in [0,+\infty)$}.
\end{equation}
Conversely, for any Herglotz vector field $G(z,t)$ of order
$d\geq 1$
 in the unit disc
there exists a unique evolution family $(\v_{s,t})$ of order
$d$  such that \eqref{main-eq} is satisfied.

Moreover for each $t>0$ fixed
\begin{equation}\label{main-eq2}
\frac{\de \v_{s,t}(z)}{\de s}=-\v_{s,t}'(z)G(z,s)
\end{equation}
for almost every $s\in (0,t)$ and all $z\in \D$.
\end{theorem}

\begin{definition}
An evolution family $(\v_{s,t})$ is called  {\sl commuting}
  if $\v_{m,n}\circ \v_{s,t}=\v_{s,t}\circ \v_{m,n}$ for all
  $0\leq s\leq t<+\infty$ and $0\leq m\leq n<+\infty$
\end{definition}

Finally we recall the definition of Loewner chains

\begin{definition}\label{DL}
A family $(f_t)_{0\leq t<+\infty}$ of holomorphic maps of the
unit disc is a {\sl Loewner chain of order $d$}, with $d\in
[1,+\infty]$,  if

\begin{enumerate}
\item[LC1.]  $f_t:\D\to\C$ is univalent for all $t\geq 0$,

\item[LC2.] $f_s(\D)\subset f_t(\D)$ for all $0\leq s < t<+\infty,$

\item[LC3.] for any compact set $K\subset\mathbb{D}$ and  all $T>0$ there exists a
non-negative function $k_{K,T}\in L^{d}([0,T],\mathbb{R})$
such that
\[
|f_s(z)-f_t(z)|\leq\int_{s}^{t}k_{K,T}(\xi)d\xi
\]
for all $z\in K$ and all $0\leq s\leq t\leq T$.
\end{enumerate}
\end{definition}

In \cite[Theorem 1.3, Theorem 4.1]{Contreras-Diaz-Pavel} it is
proved

\begin{theorem}\label{pav}
(1) For any Loewner chain $(f_t)$ of order $d\in[1,+\infty]$,
let
$$
\varphi_{s,t}:= f_t^{-1}\circ f_s, \quad 0\le s\le t.
$$
Then $(\v_{s,t})$ is an evolution family of the same order~$d$.
Conversely, for any evolution family $(\v_{s,t})$ of
order~$d\in[1,+\infty]$, there exists a Loewner chain $(f_t)$
of  order~$d$ such that
\begin{equation*}
 f_t\circ\varphi_{s,t}=f_s,\quad 0\le s\le t.
\end{equation*}
(2) Moreover let $G(z,t)$ be the Herglotz vector field of order
$d\in[1,+\infty]$ associated with the evolution family
$(\varphi _{s,t}).$ Suppose that $(f_{t})$ is a family of
univalent functions in the unit disc such that
\begin{equation}\label{low-kuf}
\frac{\partial f_{s}(z)}{\partial s}=-G(z,s)f_{s}^{\prime }(z)\qquad \text{ for every }z\in \D\text{, a.e. }s\in
\lbrack 0,+\infty ).
\end{equation}
Then $(f_{t})$ is a Loewner chain of order $d$ associated with
the evolution family~$(\varphi _{s,t}).$
\end{theorem}

We remark that, although we never use this fact in the present
paper, given any Loewner chain $(f_t)$ there exists a Herglotz
vector field such that \eqref{low-kuf} is satisfied
\cite[Theorem 4.1]{Contreras-Diaz-Pavel}.

Throughout the paper, whenever not explicitly needed, in the
statements  we simply write evolution families, Herglotz vector
fields and Loewner chains without mentioning the order.

\section{Splitting Herglotz vector fields}\label{split}

Evolution families associated with Herglotz vector fields are
of ``semigroups type'' as explained here:

\begin{proposition}\label{uno}
Let  $G(z,t)=g(t)\tilde{G}(z)$ be a splitting Herglotz vector
field. Let $(\v_{s,t})$ be the evolution family associated with
$G(z,t)$. Let $(\Phi_t)$ be the semigroup associated with
$\tilde{G}$ whose Denjoy-Wolff point is $\tau\in \oD$ and let
$h$ be the K\"onigs function of $\tilde{G}$.
 Set
\[
\lambda(t):=\int_0^t
  g(\xi)d\xi.
\]
\begin{enumerate}
  \item If $\tau\in  \D$ then
\begin{itemize}
  \item $\v_{s,t}(z)=h^{-1}(e^{\tG'(\tau)[\lambda(t)-\lambda(s)]}h(z))$,
  \item there exists a Loewner chain associated with $(\v_{s,t})$ of the form
  $f_s(z)=e^{-\tG'(\tau)\lambda(s)}h(z)$.
\end{itemize}
 \item If $\tau\in\de\D$  then
 \begin{itemize}
  \item $\v_{s,t}(z)=h^{-1}(h(z)+\lambda(t)-\lambda(s))$,
  \item there exists a Loewner chain associated with $(\v_{s,t})$ of the form  $f_s(z)=h(z)-\lambda(s)$.
\end{itemize}
\end{enumerate}
\end{proposition}

\begin{proof} Assume first that $\tau\in \D$. Set
\[
f_s(z):=e^{-\tG'(\tau)\lambda(s)} h(z).
\]
Recall that $\tG(z)=\tG'(\tau)\frac{h(z)}{h'(z)}$. Then for all
$z\in \D$ and almost every $s\in [0,+\infty)$ it follows
\begin{equation*}
\begin{split}
\frac{\de f_s(z)}{\de s}&=-\tG'(\tau)\stackrel{\bullet}{ \lambda}(s)e^{-\tG'(\tau)\lambda(s)}h(z)
=-\tG'(\tau)g(s)e^{-\tG'(\tau)\lambda(s)}h(z)
\\&=-g(s)e^{-\tG'(\tau)\lambda(s)}h'(z)\tG(z)=-G(z,s)f_s'(z).
\end{split}
\end{equation*}
Hence $\{f_s(z)\}$ is a family of univalent maps in the unit
disc which satisfies
\[
\frac{\de f_s(z)}{\de s}=-G(z,s)f'_s(z) \quad \hbox{for all $z\in\D$, a.e. $s\in [0,+\infty)$}.
\]
By Theorem \ref{pav}.(2) it follows that $\{f_s\}$ is a Loewner
chain of order $d$ associated with $G(t,z)$. In particular
$f_s(\D)\subseteq f_t(\D)$ for all $0\leq s\leq t<+\infty$ and
$\v_{s,t}(z)=f_t^{-1}\circ f_s(z)$, proving the statement.

Assume now $\tau\in\de\D$. Let
\[
f_s(z):=h(z)-\lambda(s).
\]
Recall that in this case $\tG(z)=\frac{1}{h'(z)}$. Thus, for
all $z\in \D$ and almost every $s\in [0,+\infty)$, it follows
\[
\frac{\de f_s(z)}{\de s}=-\stackrel{\bullet}{\lambda}(s)=-g(s)\tG(z)h'(z)=-G(z,s)f_s'(z),
\]
and as before we can conclude by Theorem \ref{pav}.(2).
\end{proof}

\begin{remark}
Assuming the notations and hypotheses of Proposition \ref{uno},
as a result of the statement, it follows that
$e^{\tG'(\tau)[\lambda(t)-\lambda(s)]}h(\D)\subseteq h(\D)$ in
case $\tau\in \D$ and that
$h(\D)+(\lambda(t)-\lambda(s))\subseteq h(\D)$ in case
$\tau\in\de\D$, for all $0\leq s\leq t<+\infty$ .
\end{remark}

\begin{remark}
Note that if a Herglotz vector field $G(z,t)$ has an associated
Loewner chain of the form as in Proposition \ref{uno} then from
\eqref{low-kuf} it follows at once that $G(z,t)$ is splitting.
\end{remark}

As an immediate consequence of Proposition \ref{uno} we have
part of Theorem \ref{main}

\begin{theorem}\label{m1}
Let $(\v_{s,t})$ be an evolution family with associated
Herglotz vector field $G(z,t)$. If $G(z,t)$ is splitting then
$(\v_{s,t})$ is commuting.
\end{theorem}

Proposition \ref{uno} has also the following consequence:

\begin{corollary}
Let $(\v_{s,t})$ be an evolution family of order $d\in
[1,+\infty]$ in $\D$. Then the following are equivalent:
\begin{enumerate}
  \item the Herglotz vector field $G(z,t)$ associated with $(\v_{s,t})$
  is of the form $G(z,t)=g(t)\tG(z)$ (for all $z\in \D$ and
  almost every $t\in [0,+\infty)$) for some $g\in L^d_{\sf
  loc}([0,+\infty),\R^+)$ and some infinitesimal generator
  $\tG$.
  \item There exists a semigroup $(\Phi_t)$ of holomorphic self-maps of
  $\D$ and a locally absolutely continuous and
  non-decreasing function $\lambda: \R^+\to \R^+$ with
  $\stackrel{\bullet}{\lambda}\in L^d_{\sf
  loc}([0,+\infty),\R)$, such that
  $\v_{s,t}(z)=\Phi_{\lambda(t)-\lambda(s)}(z)$.
\end{enumerate}
\end{corollary}

\begin{proof}
Assume (1). Let $h$ be the K\"onigs function of the semigroup
$(\Phi_t)$ generated by $\tG$. In case $\tau\in \D$ then
$\Phi_r(z)=h^{-1}(e^{\tG'(\tau)r}h(z))$ for all $r\geq 0$,
while, if $\tau\in\de \D$ it follows $\Phi_r(z)=h^{-1}(h(z)+r)$
for all $r\geq 0$. Let $\lambda(t):=\int_0^t g(\xi)d\xi$. Since
$g(t)\geq 0$ for almost every $t\in [0,+\infty)$ it follows
that $\lambda(t)\geq \lambda(s)$ for $0\leq s\leq t<+\infty$.
Hence, by Proposition \ref{uno} we have
$\v_{s,t}(z)=\Phi_{\lambda(t)-\lambda(s)}(z)$, hence (2) holds.

Conversely, assuming (2), let $\tG$ be the infinitesimal
generator associated with $(\Phi_t)$. Then, on the one side
\[
\frac{\de \v_{s,t}(z)}{\de t}=\frac{\de \Phi_{\lambda(t)-\lambda(s)}(z)}{\de t}=\stackrel{\bullet}{\lambda}(t)
\stackrel{\bullet}{\Phi}_{\lambda(t)-\lambda(s)}(z)=\stackrel{\bullet}{\lambda}(t)\tG(\Phi_{\lambda(t)-\lambda(s)}(z))
\]
and, on the other side by \eqref{main-eq},
\[
\frac{\de \v_{s,t}(z)}{\de t}=G(\v_{s,t}(z),t)=G(\Phi_{\lambda(t)-\lambda(s)}(z), t).
\]
Hence $G(\Phi_{\lambda(t)-\lambda(s)}(z),
t)=\stackrel{\bullet}{\lambda}(t)\tG(\Phi_{\lambda(t)-\lambda(s)}(z))$
for all $z\in \D$ and almost every $t\in [0,+\infty)$. Setting
$s=t$ for those points $s$ where $\lambda$ is differentiable we
obtain (1).
\end{proof}

Now we are going to see how the function $g(t)$ in the
decomposition of a splitting Herglotz vector field can be
according to dynamical properties of the associated evolution
family.

\begin{proposition}\label{sobreg}
Let $G(z,t)=g(t)\tilde{G}(z)$ be a splitting Herglotz vector
field. Let $(\v_{s,t})$ be the evolution family associated with
$G(z,t)$. Let $(\Phi_t)$ be the semigroup associated with
$\tG$. Suppose that either $(\Phi_t)$ is hyperbolic or there
exists  a boundary repelling fixed point for $(\Phi_t)$. Then
$g(t)\in \R$ for almost every $t\in [0,+\infty)$. Moreover
\begin{enumerate}
  \item either $g(t)\geq 0$ for almost every $t\in
  [0,+\infty)$,
  \item  or $(\Phi_{t})$ is a group of hyperbolic automorphisms  and there
  exist  $\tau,\sigma\in\de\D$, $\tau\neq\sigma$ such that
  $\tG(z)=\lambda(z-\tau)(z-\sigma)$ for some
  $\lambda\in\C$ such that $\Re
  \lambda(\sigma+\tau)=|\lambda||1+\tau\sigma|$.
\end{enumerate}
\end{proposition}

\begin{proof}
Let $\tau\in\de \D$ be either the Denjoy-Wolff point of the
semigroup in case $(\Phi_t)$ is hyperbolic, or the boundary
repelling fixed point. Then, by the very definition
$\lim_{\R\ni r\to 1^-} \Phi_t(r\tau)=\tau$ and $\lim_{\R\ni
r\to 1^-} \Phi'_t(r\tau)=e^{\beta t}$ for some
$\beta\in\R\setminus\{0\}$. Hence by \cite[Theorem
1]{Contreras-Diaz-Pommerenke:Scand} it follows that
\[
\lim_{\R\ni r\to 1^-}\tG(r\tau)=0, \quad \lim_{\R\ni r\to 1^-}\tG'(r\tau)=\beta.
\]
Since $G(z,t)=g(t)\tG(z)$, we have for almost every $t\in
[0,+\infty)$
\[
\lim_{\R\ni r\to 1^-}G(r\tau, t)=0, \quad \lim_{\R\ni r\to 1^-}G'(r\tau, t)=g(t)\beta.
\]
Since $G(z,t)$ is an infinitesimal generator for almost every
$t\in [0,+\infty)$, another application of \cite[Theorem
1]{Contreras-Diaz-Pommerenke:Scand} gives
\[
g(t)\beta\in \R
\]
for almost every $t\in [0,+\infty)$, from which it follows that
$g(t)\in \R$ for a.e. $t\in [0',+\infty)$.

Now, assume there exists $t_0\in [0,+\infty)$ such that
$G(z,t_0)$ is an infinitesimal generator and $g(t_0)<0$. Then
$G(z,t_0)=-|g(t_0)|\tG(z)$ is an infinitesimal generator. Since
infinitesimal generators form a real cone, $-\tG(z)$ is an
infinitesimal generator as well. Hence $\tG(z)$ is an
infinitesimal generator of a {\sl group} of automorphisms of
$\D$, having $\tau$ as a fixed point. Now, semigroups of
elliptic and parabolic automorphisms have only one (common)
fixed point (see, {\sl e.g.} \cite[Corollary 1.4.20]{Abate})
which must be in $\D$ in the elliptic case and
$\Phi_t'(\tau)=1$ in the parabolic case. Hence $(\Phi_t)$ is a
group of hyperbolic automorphisms. This implies that  $\tG(z)$
has the required form  by \cite[Theorem 2.3]{BCDfun}.
\end{proof}

\begin{example}
Let $\tau\in\de \D$. Let
$\tG(z)=(z-\tau)(\overline{\tau}z-1)C_\tau(z)$ with
$C_\tau(z)=(\tau+z)/(\tau-z)$ the Cayley transform with pole
$\tau$. Then $\tG(z)$ is an infinitesimal generator of a group
of hyperbolic automorphisms with Denjoy-Wolff point $\tau$ (see
\cite[Example 2.7]{BCDfun}). Let $G(z,t)=-1^{[t]}\tG(z)$, where
$[t]$ denotes the integer part of $t$. Then  $G(z,t)$ is a
splitting Herglotz vector field and for each fixed $t\geq 0$,
$G(z,t)$ generates a group of hyperbolic automorphisms of $\D$.
\end{example}

\begin{example}
Let $G(z,t):=(1+it)(z-1)^2$. Then $G(z,t)$ is a splitting
Herglotz vector field, with $\tG(z)=(z-1)^2$ and $g(t)=1+it$.
Notice that $g(t)\tG(z)$ generates a semigroup of parabolic
type with no boundary repelling fixed points for each fixed
$t\geq 0$.
\end{example}

\begin{example}
Let $G(z,t):=-(t(1+i)+1)z(2+z)$. Then $G(z,t)$ is a splitting
Herglotz vector field, with $\tG(z)=-z(2+z)$ and
$g(t)=t(1+i)+1$. Notice that $g(t)\tG(z)$ generates a semigroup
of elliptic type with no boundary repelling  fixed points for
each fixed $t\geq 0$.
\end{example}

\begin{theorem}
Let $G(z,t)=g(t)\tG(z)$ be a splitting Herglotz vector field of
order $d$ such that $\tG$ is not a generator of a group of
hyperbolic automorphisms of $\D$. Let
$\tG(z)=(z-\tau)(\overline{\tau}z-1)p(z)$ be the Berkson-Porta
decomposition of $\tG$. Then the Berkson-Porta data $( p(z,t),
\tau(t))$ of $G(z,t)$ is
\[
\tau(t)=\tau, \quad p(z,t)=g(t)p(z).
\]
\end{theorem}

\begin{proof}
At those points where $g(t)=0$ the result is true. Then we can
assume that $g(t)\neq 0$ for almost every $t$.

If $\tau\in \D$ then $\tG(\tau)=0$ and hence
$G(\tau,t)=g(t)\tG(\tau)=0$ for almost every $t\in
[0,+\infty)$. Therefore for almost every $t\in [0,+\infty)$ it
follows that
\[
G(z,t)=(z-\tau)(\overline{\tau}z-1)p(z,t).
\]
By the ``essential'' uniqueness of the Berkson-Porta data
 it follows that $\tau(t)=\tau$ and
$p(z,t)=g(t)p(z)$ for almost every $t\in [0,+\infty)$.

Assume $\tau\in\de\D$. Let $(\Phi_t)$ be the semigroup
generated by $\tG$ and let $\lim_{\R\ni r\to 1^-}
\Phi'_t(r\tau)=e^{\beta t}$. There are two cases.

If $(\Phi_t)$ is hyperbolic then $\beta<0$, and by hypothesis
and by Proposition \ref{sobreg} it follows that $g(t)\geq 0$
for almost every $t\in [0,+\infty)$ (and actually $g(t)>0$ for
almost every $t\geq 0$ because we are assuming $g(t)\neq 0$
almost everywhere). If $(\Phi_t)$ is parabolic then $\beta=0$.
In both cases
\[
\lim_{\R\ni r\to 1^-}G(r\tau, t)=0, \quad \lim_{\R\ni r\to 1^-}G'(r\tau, t)=g(t)\beta\leq 0,
\]
for almost every $t$.  By \cite[Theorem
1]{Contreras-Diaz-Pommerenke:Scand} this implies that $\tau$ is
the Denjoy-Wolff point of the semigroup generated by
$g(t)\tG(z)$ for almost every $t$. Hence by the Berkson-Porta
formula
\[
G(z,t)=(z-\tau)(\overline{\tau}z-1)p(z,t), \quad \hbox{a.e. } t\geq 0,
\]
and again by the uniqueness of the Berkson-Porta data it
follows $\tau(t)=\tau$ and $p(z,t)=g(t)p(z)$ for almost every
$t\in [0,+\infty)$.
\end{proof}

\section{Commuting evolution families}

The aim of the present section is to prove that the Herglotz
vector field of a commuting evolution family is splitting. To
this aim we need some preliminary results, interesting by
themselves.

Recall that if $X,Y:\D\to \C$ are holomorphic vector fields
then $[X,Y]:=YX'-XY'$.

It is well known that $[X,Y]\equiv 0$ if and only if their
flows are commuting. Moreover, since $T\D$ is generated by
$\frac{\de}{\de z}$ then $[X,Y]\equiv 0$ is equivalent to
$X(z)=\lambda Y(z)$ for some $\lambda\in \C$. Thus the
following holds.

\begin{lemma}\label{commutingbrac}
Let $G(z,t)$ be a Herglotz vector field. For almost every
$t\geq 0$ fixed, let $(\phi^t_r)$ be the semigroup generated by
$G(\cdot,t)$. Then the following are equivalent:
\begin{enumerate}
  \item $G(z,t)$ is splitting,
  \item $[G(\cdot,t),G(\cdot,s)]\equiv 0$
for almost every $s,t\in [0,+\infty)$,
  \item $\phi^t_r \circ \phi^s_r=\phi^s_r\circ \phi^t_r$ for all $r\geq 0$
  and almost all $t,s\geq 0$.
\end{enumerate}
\end{lemma}

The next proposition is a technical result about the
differentiability of  evolution families (cfr. \cite[Theorem
6.4]{BCM1})

\begin{proposition}\label{techdiff}
Let $G(z,t)$ be a Herglotz vector field of order $d\in
[1,+\infty]$ and let $(\v_{s,t})$ be its associated evolution
family. Then there exists a zero measure set $M\in [0,+\infty)$
such that for all $t\in [0,+\infty)\setminus M$ it holds
\[
\lim_{h\to 0^+} \frac{\v_{t,t+h}(z)-z}{h}=G(z,t),
\]
uniformly on compacta of $\D$.
\end{proposition}

\begin{proof}
Let $f^t_h(z):=\frac{\v_{t,t+h}(z)-z}{h}$ for $0<h<1$. We first
show that $\{f^t_h\}$ is a normal family for all $t\in
[0,+\infty)\setminus N_0$ for some set $N_0$ of zero measure.

To this aim, let $\{K_n\}_{n\in \N}$ be a sequence of compacta
of $\D$ such that $K_{n}\subset \stackrel{o}{K}_{n+1}$ and
$\cup_n K_n=\D$. Let $\{T_n\}_{n\in \N}$ be a sequence of
positive real number such that $T_n<T_{n+1}$ and $\lim_{n\to
\infty} T_n=+\infty$. By the very definition of evolution
family (Property EF3) it follows that for each $n$ there exists
$k_n:=k_{K_n,T_n}\in L^d([0,T_n], \R^+)$ such that
\[
|\v_{t,t+h}(z)-z|\leq \int_t^{t+h} k_n(\xi)d\xi
\]
for all $t,h\geq 0$ with $t+h<T_n$ and all $z\in K_n$. Hence
\[
|f^t_h(z)|\leq \frac{1}{h}\int_t^{t+h} k_n(\xi)d\xi,
\]
and, since the function on the right hand side tends to
$k_n(t)$ for almost every $t\in [0,T_n)$, it follows that
$\{f^t_h\}$ is a equibounded in $K_n$ for almost every $t\in
[0,T_n)$. Since the countable union of zero measure sets has
zero measure, it follows that $\{f^t_h\}$ is equibounded on
compacta of $\D$ for almost all $t\geq 0$. Montel's theorem
implies that $\{f^t_h\}$ is normal.

By \cite[Theorem 6.4]{BCM1} there exists a zero measure set
$N_1\subset [0,+\infty)$ such that for all $t\in
[0,+\infty)\setminus N_1$ the limit
\[
\lim_{h\to 0^+} \frac{\v_{0,t+h}(z)-\v_{0,t}(z)}{h}=G(\v_{0,t}(z), t)
\]
uniformly on compacta of $\D$. Hence for all $t\in
[0,+\infty)\setminus N_1$ it follows
\[
\lim_{h\to 0^+}\frac{\v_{t,t+h}(\v_{0,t}(z))-\v_{0,t}(z)}{h}=
\lim_{h\to 0^+}\frac{\v_{0,t}(z)-\v_{0,t}(z)}{h}=G(\v_{0,t}(z),t).
\]
Now, let $t\in [0,+\infty)\setminus (N_0\cup N_1)$. Let $f^t:\D
\to \C\cup\{\infty\}$ be any limit of $\{f_h^t\}$. By the
previous equation it follows that $f^t(\cdot)=G(\cdot,t)$ on
the open set $\v_{0,t}(\D)$. Therefore $f^t(\cdot)=G(\cdot,t)$
on $\D$ and this proves the theorem with $M=N_0\cup N_1$.
\end{proof}

Now we are in good shape to prove the remaining part of Theorem
\ref{main}:

\begin{theorem}\label{maindue}
Let $G(z,t)$ be a Herglotz vector field of order $d\in
[1,+\infty]$ and let $(\v_{s,t})$ be its associated evolution
family. If $(\v_{s,t})$ is commuting then $G(z,t)$ is
splitting.
\end{theorem}

\begin{proof}
Let $g^t_h(z):=\v_{t,t+h}(z)$ for $t\geq 0$ and $0\leq h\leq
1$. Then $g^t_h\to \id$ as $h\to 0$. Moreover, by Proposition
\ref{techdiff} it follows
\[
\lim_{h\to 0^+} \frac{g_h^t(z)-z}{h}=G(z,t),
\]
uniformly on compacta of $\D$ for almost all $t\geq 0$.

Let $(\phi^t_r)$ be the semigroup associated with $G(\cdot,t)$
for $t\geq 0$ fixed (this is well defined for almost all $t\geq
0$). By the product formula (see \cite[Theorem 6.12]{R-S2}) we
have
\[
\phi^t_r=\lim_{n\to \infty} (g^t_{r/n})^{\circ n},
\]
where the limit is uniform on compacta of $\D$. Hence, for
almost all $s\neq t$ and for all $r\geq 0$ we have
\[
\phi_r^t \circ \phi_r^s =\lim_{m,n\to \infty} (g^t_{r/n})^{\circ n}\circ (g^s_{r/m})^{\circ m}
= \lim_{m,n\to \infty} (g^s_{r/n})^{\circ n}\circ (g^t_{r/m})^{\circ m}=\phi_r^s \circ \phi_r^t.
\]
Lemma \ref{commutingbrac} implies that $G(z,t)$ is splitting.
\end{proof}

\section{Reversing evolution families}\label{rever}

\begin{definition}
An evolution family $(\v_{s,t})$ is called {\sl reversing} if
$\v_{s,t}=\v_{s,u}\circ \v_{u,t}$ for all
  $0\leq s\leq u\leq t<+\infty$.
\end{definition}

\begin{remark}
Note that if $(\v_{s,t})$ is a commuting evolution family then
it is also reversing because $\v_{u,t}\circ
\v_{s,u}=\v_{s,t}=\v_{s,u}\circ \v_{u,t}$ for all $0\leq s\leq
u\leq t<+\infty$.
\end{remark}

Now we study common fixed points of  reversing evolution
families. First, we show that, although in principle a
reversing family is not commuting, one can always find a finite
chain of  mappings such that each commutes with the previous
one, relating any two elements of the family.

\begin{lemma}\label{commutafam}
Let $(\v_{s,t})$ be a reversing evolution family. Then for any
$0\leq u\leq v<+\infty$ and $0\leq s\leq t<+\infty$ such that
$\v_{s,t}\neq \id$ and $\v_{u,v}\neq \id$ there exist $1\leq
m\leq 4$ and $\{(s_0,t_0),\ldots, (s_m,t_m)\}$ such that
$s_0=u$, $t_0=v$, $s_m=s$, $t_m=t$, $0\leq s_j\leq
t_j<+\infty$, $\v_{s_j,t_j}\neq \id$ for all $j=0,\ldots, m$,
and
\[
\v_{s_j,t_j}\circ
\v_{s_{j+1},t_{j+1}}=\v_{s_{j+1},t_{j+1}}\circ\v_{s_j,t_j}
\]
for $j=0,\ldots, m-1$.
\end{lemma}

\begin{proof}
Let $0\leq u\leq v<+\infty$ and $0\leq s\leq t<+\infty$ be such
that $\v_{u,v}$ and $\v_{s,t}$ are not the identity and they do
not commute (otherwise the result is true with $m=1$). We can
assume that $v\leq t$. First, let $v< t$.

By hypothesis of reversing, $\v_{l,r}$ commutes with $\v_{r,n}$
for all $0\leq l\leq r\leq n<+\infty$.

Hence $\v_{u,v}$ commutes with  $\v_{v,n}$ for $n\geq v$. In
particular, it commutes with $\v_{v,t}$. Suppose that
$\v_{v,t}\neq \id$. Then $\v_{v,t}$ commutes with $\v_{t,r}$
for all $r\geq t$. Also, $\v_{s,t}$ commutes with $\v_{t,r}$
for all $r\geq t$. If $\v_{t,r}=\id$ for all $t\leq r$ then by
\eqref{main-eq} it follows that
\begin{equation}\label{Gcero}
0=\frac{\de \v_{t,r}(z)}{\de r}=G(\v_{t,r}(z),r)=G(z,r)
\end{equation}
for almost every $r\geq t$. Hence $G(z ,r)\equiv 0$ for almost
every $r\geq t$. Therefore, again by \eqref{main-eq},
$\v_{s,r}=\id$ for all $r\geq t$, hence $\v_{s,t}=\id$,
contradicting our hypothesis. Therefore, if $\v_{v,t}\neq \id$
the result is proven with $m=3$.

Assume that $\v_{v,t}=\id$. We claim that there exists $v<t'<t$
such that $\v_{v,t'}\neq \id$, $\v_{t',t}\neq \id$. If this is
the case, then $\v_{u,v}$ commutes with $\v_{v,t'}$ which
commutes with $\v_{t',t}$ which commutes with $\v_{t,r}$ for
all $r\geq t$ and such elements---which cannot be all
$\equiv\id$ as we saw before---commute with $\v_{s,t}$,
concluding the result with $m=4$.

We need to show that we can choose $v<t'<t$ such that
$\v_{v,t'}\neq \id$, $\v_{t',t}\neq \id$. Indeed, arguing by
contradiction, let $N:=\{r\in[v,t]:\v_{v,r}=\id\}$ and
$M:=[v,t]\setminus N$. Thus $\v_{r,t}=\id$ for all $r\in M$.
Hence $\frac{\de \v_{r,t}}{\de r}=0$ for almost all $r\in M$
and $\frac{\de \v_{v,r}}{\de r}=0$ for almost all $r\in N$. The
last condition, as in \eqref{Gcero}, implies that $G(z,r)\equiv
0$ for almost every $r\in N$. By \eqref{main-eq2} we have also
that for almost every $r\in M$
\[
0=\frac{\de \v_{r,t}}{\de r}=-\v_{r,t}'(z) G(z,r)=-G(z,r),
\]
hence $G(z,r)\equiv 0$ for almost every $r\in M$. Since
$[v,t]=M\cup N$, we have $G(z,r)\equiv 0$ for almost every
$r\in [v,t]$. Therefore, by \eqref{main-eq}, $\v_{s,r}=\id$ for
all $r\in [v,t]$. Thus $\v_{s,t}=\id$ against our hypothesis.

Finally, the case $v=t$ follows easily by noting that
$\v_{u,t}, \v_{s,t}$ commute with $\v_{t,r}$ for any $r\geq t$.
\end{proof}

The previous lemma has several interesting consequences. We
start with the following result about hyperbolic automorphisms:

\begin{proposition}\label{coro-iper-grup}
Let $(\v_{s,t})$ be a reversing evolution family such that
$\v_{u,v}$ is a hyperbolic automorphism of $\D$ for some $0\leq
  u< v<+\infty$. Then for all $0\leq
  s< t<+\infty$ with $\v_{s,t}\neq\id$ it follows
that $\v_{s,t}$ is a hyperbolic automorphism of $\D$. Moreover
if $G(z,t)$ is the associated Herglotz vector field of
$(\v_{s,t})$  then $G(z,t)$ is splitting and the family is
commuting. In
  particular, there exist two distinct points $\tau,\tau'\in\de
  \D$ such that $\v_{s,t}(\tau)=\tau, \v_{s,t}(\tau')=\tau'$
  for all $0\leq s\leq t<+\infty$.
\end{proposition}

\begin{proof}
Let $\v_{s,t}\neq \id$. Let $\{\v_{s_0,t_0},\ldots,
\v_{s_m,t_m}\}$ be a chain  such that $\v_{s_j,t_j}\neq \id$
for all $j=0,\ldots, m$, $\v_{s_0,t_0}=\v_{u,v}$,
$\v_{s_m,t_m}=\v_{s,t}$ and $\v_{s_j,t_j}\circ
\v_{s_{j+1},t_{j+1}}=\v_{s_{j+1},t_{j+1}}\circ\v_{s_j,t_j}$ for
$j=0,\ldots, m-1$. By Lemma \ref{commutafam} such a chain
exists. By Lemma \ref{centro}.(1) $\v_{s,t}$ commutes with
$\v_{u,v}$ and it is a hyperbolic automorphism with the same
fixed points. By Theorem \ref{maindue} the associated Herglotz
vector field is splitting. We note that one can even prove
directly the last assertion without applying Theorem
\ref{maindue}. In fact, moving to the right half plane by means
of a conjugation with a Cayley transform, one sees that all the
elements of the evolution family are of the form
$\lambda(s,t)w$, and hence by \eqref{main-eq}, we see that
$G(z,t)$ is splitting.
\end{proof}

In case there are no hyperbolic automorphisms in a reversing
family we get:

\begin{proposition}\label{comrev}
Let $(\v_{s,t})$ be a reversing evolution family. Suppose that
$\v_{s,t}$ is not a hyperbolic automorphism of $\D$ for all
  $0\leq s\leq t<+\infty$. Then there exists $\tau\in\oD$
such that $\tau$ is the Denjoy-Wolff point of $\v_{s,t}$ for
all $0\leq s\leq t<+\infty$ with $\v_{s,t}(z)\not\equiv z$.
\end{proposition}

\begin{proof}
Let $\v_{s,t}, \v_{u,t}\neq \id$. Let $\{\v_{s_0,t_0},\ldots,
\v_{s_m,t_m}\}$ be a chain  such that $\v_{s_j,t_j}\neq \id$
for all $j=0,\ldots, m$, $\v_{s_0,t_0}=\v_{u,v}$,
$\v_{s_m,t_m}=\v_{s,t}$ and $\v_{s_j,t_j}\circ
\v_{s_{j+1},t_{j+1}}=\v_{s_{j+1},t_{j+1}}\circ\v_{s_j,t_j}$ for
$j=0,\ldots, m-1$. By Lemma \ref{commutafam} such a chain
exists with $m\leq 4$. By Behan's theorem \cite{Behan} it
follows that either $\v_{s_j,t_j}, \v_{s_{j+1},t_{j+1}}$ are
hyperbolic automorphisms or they share the same Denjoy-Wolff
point. By hypothesis there are no hyperbolic automorphisms and
hence the result is proved.
\end{proof}

Next we show that in many cases a reversing evolution family is
commuting:

\begin{theorem}\label{revcom}
Let $(\v_{s,t})$ be a reversing evolution family. Suppose that
one of the following holds:
\begin{enumerate}
  \item there exists $0\leq u\leq v<+\infty$ such that
  $\v_{u,v}$ is elliptic,
  \item there exists $0\leq u\leq v<+\infty$ such that
  $\v_{u,v}$ is hyperbolic,
  \item  for all $0\leq u\leq v<+\infty$ such that $\v_{u,v}\neq \id$, the maps $\v_{u,v}$
  are parabolic of zero hyperbolic step.
\end{enumerate}
Then $(\v_{s,t})$ is commuting.
\end{theorem}

\begin{proof}
In case $\v_{u,v}$ is a hyperbolic automorphism the result
follows from Proposition \ref{coro-iper-grup}. In case  all
$\v_{s,t}\neq \id$ are elliptic automorphisms, by Proposition
\ref{comrev}, there exists $\tau\in \D$ which is a common fixed
point for all the family. Thus, up to conjugation with a fixed
automorphism which maps $\tau$ to $0$, we can assume that
$\tau=0$. Hence $\v_{s,t}(z)=\lambda(s,t) z$ for some
$|\lambda(s,t)|=1$ and the family is commuting.

We can assume that $\v_{u,v}$ is not a hyperbolic or elliptic
automorphism (but note that we are not excluding it can be a
parabolic automorphism). Let $\v_{s,t}\neq \id$. Let
$\{\v_{s_0,t_0},\ldots, \v_{s_m,t_m}\}$ be a chain of minimal
length such that $\v_{s_j,t_j}\neq \id$ for all $j=0,\ldots,
m$, $\v_{s_0,t_0}=\v_{u,v}$, $\v_{s_m,t_m}=\v_{s,t}$ and
$\v_{s_j,t_j}\circ
\v_{s_{j+1},t_{j+1}}=\v_{s_{j+1},t_{j+1}}\circ\v_{s_j,t_j}$ for
$j=0,\ldots, m-1$. By Lemma \ref{commutafam} such a chain
exists with $m\leq 4$. By Lemma \ref{centro} it must be $m=1$,
hence $\v_{s,t}$ commutes with $\v_{u,v}$. Thus
$(\v_{s,t})\subset C(\v_{u,v})$, the centralizer of $\v_{u,v}$.
Again by Lemma \ref{centro} such a centralizer is abelian, and
hence the family is indeed commuting.
\end{proof}

\begin{remark}
Note that if a reversing evolution family $(\v_{s,t})$ contains
a parabolic element $\v_{u,v}$ then for all $0\leq s\leq
t<+\infty$ such that $\v_{s,t}\neq \id$ it follows that
$\v_{s,t}$ is parabolic (but the hyperbolic step can be zero or
positive). This follows at once by Lemma \ref{commutafam} and
\cite[Corollary 4.1]{Cow}.
\end{remark}

Theorem \ref{revcom} together with Theorem \ref{main} implies
that the Herglotz vector field of a reversing evolution family
satisfying the hypothesis of Theorem \ref{revcom}  is
splitting. In the elliptic and hyperbolic cases, such a result
can be proved directly by looking at the Herglotz vector field.
We provide here such a proof, which can be  also extended to
the parabolic case at the price of assuming some regularity for
the vector field.

First, we relate reversing with a property of the Herglotz
vector field.

\begin{lemma}\label{revinG}
Let $(\v_{s,t})$ be an evolution family associated with the
Herglotz vector field $G(z,t)$. Then the following conditions
are equivalent
\begin{enumerate}
  \item $(\v_{s,t})$ is reversing.
  \item $G(\v_{s,t}(z),u)=\v_{s,t}'(z)G(z,u)$ for every $s,t$ and almost every $u$ such that $0\leq s\leq u\leq
  t<+\infty$.
\end{enumerate}
\end{lemma}

\begin{proof}
Assume (1) is satisfied. Fixing $s,t$ and differentiating with
respect to $u$ the equation $\v_{s,t}=\v_{s,u}\circ \v_{u,t}$,
using \eqref{main-eq} and \eqref{main-eq2} we obtain for almost
every $u$
\begin{equation*}
\begin{split}
0  &=\frac{\de\v_{s,u}}{\de u}(\v_{u,t}(z))+\v_{s,u}'(\v_{u,t}(z))\frac{\de \v_{u,t}(z)}{\de u}
\\&=G(\v_{s,u}(\v_{u,t}(z)),u)-\v_{s,u}'(\v_{u,t}(z))\v_{u,t}'(z) G(z,u)\\
&=G(\v_{s,u}\circ \v_{u,t}(z), u)- (\v_{s,u}\circ \v_{u,t})'(z)G(z,u)\\
&=G(\v_{s,t}(z),u)-\v_{s,t}'(z)G(z,u).
\end{split}
\end{equation*}
Thus (2) holds.

Conversely, assume (2) holds. Fix $z\in \D$ and $0\leq s\leq
t<+\infty$. For $s\leq u\leq t$  let
$f(u):=\v_{s,t}(z)-\v_{s,u}\circ \v_{u,t}(z)$. Note that $f$ is
absolutely continuous by \cite[Proposition 3.7]{BCM1} and
$f(s)=0$. Differentiating $f$ with respect to $u$, by the
previous computations, we obtain that $f'(u)=0$ almost
everywhere. Thus $f\equiv 0$ and (1) holds.
\end{proof}

\begin{proposition}\label{rev-Gsplit}
Let $(\v_{s,t})$ be a reversing evolution family with
associated Herglotz vector field $G(z,t)$. Let $0\leq u\leq
v<+\infty$ be such that $\v_{u,v}\neq \id$. Assume that
$\v_{u,v}$ is  elliptic or hyperbolic. Then $G(z,t)$ is
splitting.
\end{proposition}

\begin{proof}
If $\v_{u,v}$ is a hyperbolic automorphism then the result
follows from Proposition \ref{coro-iper-grup}. In case
$\v_{s,t}\neq \id$ are all elliptic automorphisms then one can
argue  as in the proof of Theorem \ref{revcom}.

We can thus assume that $\v_{u,v}$ is not an automorphism.  Let
$\tau\in \oD$ be the Denjoy-Wolff point of $\v_{u,v}$. By
Proposition \ref{comrev}, it follows $\tau$ is the Denjoy-Wolff
point of $\v_{s,t}$ for all $0\leq s\leq t<+\infty$ such that
$\v_{s,t}\neq \id$. By \cite[Theorem 6.7]{BCM1} it follows then
that for almost every $t\geq 0$
\[
G(z,t)=(z-\tau)(\overline{\tau}z-1)p(z,t),
\]
with $\Re p(z,t)\geq 0$. In particular, by Berkson-Porta
formula, the semigroup generated by $G(\cdot,t)$ has
Denjoy-Wolff point $\tau$.

\noindent{\bf Claim.} There exists $t_0\geq 0$  such that
$G(\cdot,t_0)$ is an infinitesimal generator with
$G(z,t_0)\not\equiv 0$ and $p(\tau,t_0)\neq 0$ in case $\tau\in
\D$, or the angular limit
\[
\beta(t_0):=\angle\lim_{z\to \tau}\frac{G(z,t_0)}{z-\tau}
\]
exists and it is different from $0$ in case $\tau\in\de \D$.

In case $\tau\in \D$ this holds for any $t_0\geq 0$ such that
$G(z,t_0)\not\equiv 0$ is an infinitesimal generator with
Denjoy-Wolff point $\tau$ (namely for almost every $t_0\geq 0$
such that $G(z,t_0)\not\equiv 0$) for otherwise $p(\tau,t_0)=0$
which implies actually $p(z,t_0)\equiv 0$ and hence
$G(z,t_0)\equiv 0$.

In case $\tau\in\de \D$, by \cite[Theorem
1]{Contreras-Diaz-Pommerenke:Scand}  for almost every $t\geq 0$
it follows that
\begin{equation}\label{limGf}
\angle\lim_{z\to \tau} \frac{G(z,t)}{z-\tau}=\beta(t) \in (-\infty, 0].
\end{equation}
By hypothesis  $\v_{u,v}'(\tau)<1$ (in the sense of angular
limits). According to \cite[Theorem 7.1]{BCM1},
$\v_{u,v}'(\tau)=\exp(\lambda(s)-\lambda(t))$ with
\[
\lambda(t)=\int_{0}^{t}\left(  \angle\lim_{z\rightarrow\tau}\frac
{2|\tau-z|^{2}p(z,\xi)}{1-|z|^{2}}\right)  d\xi.
\]
Hence, since $z\to \tau$ non-tangentially, if it were
$\beta(t)=0$ for almost every $t\geq 0$ it would follows that
$\lambda\equiv 0$ and $\v_{u,v}'(\tau)=1$, a contradiction.
Thus the claim is proven.

Now let us fix $t\geq 0$  such that $G(\cdot,t)$ is an
infinitesimal generator with Denjoy-Wolff point $\tau$ (this
happens for almost every $t\geq 0$ such that $G(z,t)\not\equiv
0$). Consider the function
\[
A(z):=\frac{G(z,t)}{G(z,t_0)}.
\]
Then in case $\tau \in \D$
\begin{equation}\label{limA}
\lim_{z\to \tau}A(z)=\lim_{z\to \tau} \frac{G(z,t)}{G(z,t_0)}=\frac{p(\tau,t)}{p(\tau,t_0)}.
\end{equation}
While, in case $\tau\in \de \D$, by \eqref{limGf} and since
$\beta(t_0)<0$,
\begin{equation}\label{limA2}
\angle\lim_{z\to \tau}A(z)=\angle\lim_{z\to \tau} \frac{G(z,t)}{z-\tau}\frac{z-\tau}{G(z,t_0)}=\frac{\beta(t)}{\beta(t_0)}.
\end{equation}
Now the family is reversing, therefore by Lemma \ref{revinG} we
have
\[
A(z)=\frac{\v'_{u,v}(z)G(z,t)}{\v'_{u,v}(z)G(z,t_0)}=\frac{G(\v_{u,v}(z),t)}{G(\v_{u,v}(z),t_0)}=A(\v_{u,v}(z)).
\]
By induction then
\begin{equation}\label{A-v}
    A(z)=A(\v_{u,v}^{\circ n}(z))
\end{equation}
for all $z\in \D$. But $\lim_{n\to \infty}\v_{u,v}^{\circ
n}(z)=\tau$ being $\tau$ the  Denjoy-Wolff point of $\v_{u,v}$
and, in case $\tau\in \de \D$, the sequence $\{\v_{u,v}^{\circ
n}(z)\}$ converges to $\tau$ non-tangentially (see for instance
\cite{Br1}). Thus by \eqref{A-v} and either \eqref{limA} in
case $\tau\in \D$ or \eqref{limA2} in case $\tau\in\de\D$
\[
A(z)=\lim_{n\to \infty}A(\v_{u,v}^{\circ n}(z))=g(t)
\]
with $g(t):=\frac{p(\tau,t)}{p(\tau,t_0)}$ in case $\tau \in
\D$ and $g(t)=\frac{\beta(t)}{\beta(t_0)}$ in case
$\tau\in\de\D$. Thus $G(z,t)=g(t) G(z,t_0)$ for almost every
$t\geq 0$. Hence $G(z,t)$ is splitting.
\end{proof}

The previous proof can be adapted to the parabolic case in the
following way:

\begin{theorem}\label{revcom2}
Let $(\v_{s,t})$ be a reversing evolution family with common
Denjoy-Wolff point $\tau\in \de\D$. Let $G(z,t)$ be its
associated Herglotz vector field. Suppose that $G(\cdot,t)$ has
derivatives up to order three at $z=\tau$ for almost every
$t\in [0,+\infty)$. Then $G(z,t)$ is splitting.
\end{theorem}

\begin{proof}
If the evolution family is not trivial, we can find $t\geq 0$
such that $G(\cdot,t)\not\equiv 0$, $G(\cdot,t)$ is an
infinitesimal generator and it is differentiable up to the
third order at $\tau$. We claim that there exists $\beta(t)\in
\{1,2,3\}$ such that
\[
\lim_{z\to \tau} \frac{G(z,t)}{(z-\tau)^{\beta(t)}}\neq 0.
\]
Indeed, if $\lim_{z\to \tau} \frac{G(z,t)}{(z-\tau)^3}=0$ by
the Shoikhet version of Burns-Krantz type theorem for
semigroups \cite{Sh} it follows $G(z,t)\equiv 0$.

Let $\beta=\inf \beta(t)$, where $t$ is chosen among those
$t\geq 0$ such that $G(z,t)\not\equiv 0$, $G(z,t)$ is an
infinitesimal generator and $G(z,t)$ is differentiable up to
order three at $\tau$. Let $t_0$ be such that
$\beta(t_0)=\beta$.

As in the proof of Proposition \ref{rev-Gsplit}, for almost
every $t\geq 0$ fixed, we can define the function
$A(z):=G(z,t)/G(z,t_0)$. Thus
\[
\lim_{z\to \tau} A(z)=\lim_{z\to \tau} \frac{G(z,t)}{(z-\tau)^\beta}\frac{(z-\tau)^\beta}{G(z,t_0)}=C(t)
\]
exists. Now one can argue  exactly as in the proof of
Proposition \ref{rev-Gsplit}.
\end{proof}

{\bf Question:} is there an example of a reversing evolution
family which is not commuting?

\medbreak

Such a family, if exists, should be of parabolic type and
contains parabolic mappings of positive hyperbolic step,
moreover, the associated Herglotz vector field should not be
differentiable at the Denjoy-Wolff point.

\end{document}